\documentclass{amsart}

\usepackage{amsfonts, amsmath, amssymb, amsthm}
\usepackage{mathrsfs}
\usepackage{array, multirow}
\usepackage{tikz}
\usetikzlibrary{positioning, scopes}

\usepackage[margin=1.2in]{geometry}

\newcommand{\real}{\mathbb{R}}

\newcommand{\integer}{\mathbb{Z}}

\newtheorem{theorem}{Theorem}[section]
\newtheorem{lemma}[theorem]{Lemma}
\newtheorem{corollary}[theorem]{Corollary}

\newtheorem{proposition}[theorem]{Proposition}
\newtheorem{definition}[theorem]{Definition}

\newtheorem*{proposition*}{Proposition}

\theoremstyle{definition} 
\newtheorem{remark}[theorem]{Remark}

\begin{document}

\title[Mean curvature and closed geodesics]{Mean curvature and closed geodesics in convex hypersurfaces} 
\author[J. DIBBLE]{JAMES DIBBLE}
\address{Department of Mathematics and Statistics, University of Southern Maine, Portland, ME 04101}
\email{james.dibble@maine.edu}

\author[J.A. HOISINGTON]{JOSEPH ANSEL HOISINGTON}
\address{Department of Mathematics, Rose--Hulman Institute of Technology, Terre Haute, IN 47803}
\email{hoisingt@rose-hulman.edu}


\date{}

\begin{abstract}
We give a sharp lower bound for the total mean curvature of a convex hypersurface in Euclidean space in terms of the length of a shortest nontrivial closed geodesic, generalizing a result of \'{A}lvarez Paiva for convex surfaces. This result is based on a sharp lower bound for the mean width of a convex hypersurface in terms of its Birkhoff invariant, which gives sharp lower bounds for a broader array of total curvature functionals. We also characterize spheres as the unique convex hypersurfaces whose planar sections containing chords of maximal length are all as long as possible. 
\end{abstract}

\maketitle


\section{Introduction}
\label{introduction}


This paper establishes the following sharp lower bound for the total mean curvature of a convex hypersurface in Euclidean space in terms of the length of its shortest closed geodesic, generalizing work of \'{A}lvarez Paiva \cite{AlvarezPaiva97} for convex surfaces, which we discuss below.  

\begin{theorem}\label{mean_curvature_theorem}
Let $M^{n}$ be a smooth, closed, convex hypersurface in $\real^{n+1}$, $n \geq 2$, $\mathcal{H}$ the mean curvature of $M$ (i.e., the average of its principal curvatures), $\Lambda(M)$ the length of a shortest nontrivial closed geodesic in $M$, and $\sigma_{n}$ the surface area of the unit $n$-sphere. Then, 
\begin{align}\label{mean_curvature_thm_eqn}
\displaystyle \int_{M} \mathcal{H}^{n-1} dA \geq \frac{\sigma_{n}}{2\pi}  \Lambda(M). 
\end{align}
Equality holds if and only if $M$ is a round sphere. 
\end{theorem}

\noindent The proof of Theorem \ref{mean_curvature_theorem} is based on the following sharp lower bound for the mean width of a convex hypersurface.  

\begin{theorem}\label{mean_width_theorem}
Let $M^{n}$ be a closed, convex hypersurface in $\real^{n+1}$, $n \geq 2$, $\Xi(M)$ its mean width, as in Definition \ref{width_defn}, and $\beta(M)$ its Birkhoff invariant, as in Definition \ref{birkhoff invariant}. Then,
\begin{align}\label{mean_width_thm_eqn}
\displaystyle \Xi(M) \geq \frac{1}{\pi} \beta(M).  
\end{align}
Equality holds if and only if $M$ is a round sphere. 
\end{theorem}

\noindent Theorem \ref{mean_width_theorem} is connected to Theorem \ref{mean_curvature_theorem} by two facts: the mean width of a smooth convex hypersurface can be calculated in terms of its principal curvatures (see \cite{Schneider1993} and \eqref{width_curvature_identity} below); and, by an argument due to Birkhoff, a smooth, closed, convex hypersurface contains a closed geodesic whose length is equal to its Birkhoff invariant (\cite{Birkhoff1917}, see also \cite{Birkhoff1927,Croke1988}). By means of these results, the proof of Theorem \ref{mean_width_theorem} gives sharp lower bounds for a broader class of total curvature functionals in terms of the minimum length of a nontrivial closed geodesic in smooth convex hypersurfaces, recorded in Proposition \ref{main_curvature_prop}.

The proof of the equality conditions in Theorems \ref{mean_curvature_theorem} and \ref{mean_width_theorem} is based on what is, to the best of the authors' knowledge, a novel characterization of spheres among convex hypersurfaces: they are the unique closed, convex hypersurfaces whose intersections with $2$-planes containing diametrically opposed points are all as long as possible. 

\begin{theorem}\label{slicing property implies round}
Let $M^n$ be a closed, convex hypersurface in $\real^{n+1}$, $n \geq 2$, and $w^{+}(M)$ its extrinsic diameter (as in Definition \ref{width_defn}). Suppose every affine $2$-plane containing points $x,y \in M$ with $|y-x| = w^{+}(M)$ intersects $M$ in a curve of length $\pi w^{+}(M)$. Then, $M$ is a round sphere.
\end{theorem}

\noindent Theorem \ref{slicing property implies round} complements rigidity results for convex hypersurfaces discussed in \cite[Chapter 9]{MartiniMontejanoOliveros2019}. For example, Montejano \cite{Montejano1991} proved that every convex hypersurface whose $k$-sections through a point all have constant width must be round, and S\"{u}ss \cite{Suss1947} showed that, for $n = 2$, only the sphere has congruent planar sections in each direction. 

\'{A}lvarez Paiva proved the inequality in Theorem \ref{mean_curvature_theorem} for $n=2$ in \cite{AlvarezPaiva97}. The proof there can be adapted to prove the inequality in Theorem \ref{mean_width_theorem} in all dimensions. A non-sharp version of the inequality in Theorem \ref{mean_curvature_theorem} for $n=2$ also follows from Minkowski's inequality between the mean curvature of a convex surface and its area \cite{Minkowski1903} and results of Treibergs and Croke bounding the area of a convex surface below in terms of the length of its shortest closed geodesic \cite{Treibergs1985,Croke1988}. None of these arguments or results, however, gives a description of the equality conditions in Theorems \ref{mean_curvature_theorem} and \ref{mean_width_theorem}, even for smooth convex surfaces in $\real^{3}$. The equality conditions in Theorems \ref{mean_curvature_theorem} and \ref{mean_width_theorem} admit alternate proofs under some additional assumptions, which we discuss in Remarks \ref{blaschke_remark} and \ref{main_thm_remark}, but these draw on results such as the proof of the Blaschke conjecture \cite{Besse1978}. It therefore seems that the characterization of equality in \'{A}lvarez Paiva's result, as in Theorems \ref{mean_curvature_theorem} and \ref{mean_width_theorem}, requires an argument like the one given here. 

Theorem \ref{mean_curvature_theorem} is similar in spirit to results of Willmore and Chen \cite{Willmore1965,Chen1971b} that give lower bounds for the integral of the $n$-th power of the mean curvature of a hypersurface in $\real^{n+1}$. For hypersurfaces of dimension $n \geq 3$, however, the authors do not know of general results bounding the integral of the mean curvature itself from below in terms of the length of the shortest closed geodesic or similar geometric invariants, with the exception of a result for convex $3$-spheres discussed in Remark \ref{main_thm_remark}. Theorem \ref{mean_width_theorem} gives a lower bound for the mean width of a smooth convex hypersurface in terms of the length of the shortest closed geodesic, recorded in \eqref{mcp_pf_eqn_1}, because of Birkhoff's result cited above, but the Birkhoff invariant of a general convex hypersurface does not seem to admit a geometrically natural representative---generic convex polyhedra, for example, have no simple closed geodesics \cite{Gruber1991, Galperin2003}. In this sense, Theorem \ref{mean_width_theorem} seems to be the simplest geometrically natural statement of its type. 



\subsection*{Outline of the Paper} In Section \ref{variation_section}, we establish versions of the first variation formula and geodesic equation in $C^{1,1}$ hypersurfaces. These are used in the proof of Theorem \ref{slicing property implies round}, which is given in Section \ref{slices_section}. Theorem \ref{mean_width_theorem} is proved in Section \ref{mean_width_section}, and Theorem \ref{mean_curvature_theorem} is proved in Section \ref{curvature_section}. The proofs of Theorems \ref{mean_width_theorem} and \ref{slicing property implies round} draw on several elementary results from convex geometry. For completeness, we include these results and their proofs in Sections \ref{slices_section} and \ref{mean_width_section}. 

\subsection*{Acknowledgments} The authors are grateful to Matteo Raffaelli for comments on a draft of this paper and Andreas Bernig for input about the references.


\section{The geodesic equation for $C^{1,1}$ hypersurfaces}
\label{variation_section}


The proof of Theorem \ref{slicing property implies round} uses a form of the geodesic equation satisfied by geodesics in $C^{1,1}$ hypersurfaces in $\real^{n+1}$. This is recorded in Lemma \ref{acceleration normal to hypersurface}. We begin with some basic definitions and formulas. The \textbf{length} $L(\gamma)$ of a continuous path $\gamma : [a,b] \to \real^{n+1}$, which may be infinite, is $\sup \Big\{ \sum_{i = 1}^m |\gamma(x_i) - \gamma(x_{i-1})| \;\Big|\; a = x_0 < x_1 < \cdots < x_m = b \Big\}$. A path is \textbf{rectifiable} if its length is finite, in which case its unit-speed parameterization is Lipschitz continuous and, consequently, absolutely continuous. When $\gamma$ is absolutely continuous, its length is $\int_a^b |\gamma'(t)| dt$, and its energy $E(\gamma)$ is $\frac{1}{2} \int_a^b |\gamma'(t)|^2 dt$. These are related by the Cauchy--Schwarz inequality: $L(\gamma)^2 \leq 2(b-a) E(\gamma)$, with equality if and only if $\gamma$ has constant speed. The \textbf{distance} between two points in a connected $C^{1,1}$ hypersurface $M$ in $\real^{n+1}$ is the infimum of the lengths of paths in $M$ connecting them. 

\begin{lemma}[First variation formula]\label{first variation formula}
Let $V : [c,d] \times [a,b] \to \real^{n+1}$ be a $C^{1,1}$ map, and for each $s \in [c,d]$, define the path $\gamma_s : [a,b] \to \real^{n+1}$ by $\gamma_s(t) = V(s,t)$. Then, for almost all $s_0$,
\begin{equation}\label{fvf}
\frac{d}{ds}\Big|_{s=s_0} E(\gamma_s) = \langle W(s_0,t), \gamma_{s_0}'(t) \rangle \Big|_a^b - \int_a^b \langle W(s_0,t), \gamma_{s_0}''(t) \rangle \,dt
\end{equation}
where $W = \frac{\partial V}{\partial s}$ is the variation field of $V$, $\gamma_{s_0}''$ is defined for almost all $t$ and essentially bounded, and the boundary difference $\langle W(s_0,t), \gamma_{s_0}'(t) \rangle \Big|_a^b = \langle W(s_0,b), \gamma_{s_0}'(b) \rangle - \langle W(s_0,a), \gamma_{s_0}'(a) \rangle$ uses one-sided derivatives of $\gamma$ at the endpoints.
\end{lemma}

\begin{proof} Since $V$ is $C^{1,1}$, Rademacher's theorem implies it is twice differentiable almost everywhere on $[c,d] \times [a,b]$; consequently, for almost all $s_0$, $V$ is twice differentiable at $(s_0,t)$ for almost all $t$. By the dominated convergence theorem, at each such $s_0$, one has
\[
\frac{d}{ds} \Big|_{s = s_0} E(\gamma_s) = \frac{1}{2} \frac{d}{ds} \Big|_{s = s_0} \int_a^b  |\gamma_s'(t)|^2 \,dt = \int_a^b \langle \frac{\partial^2 V}{\partial s \partial t}(s_0,t), \gamma_{s_0}'(t) \rangle \,dt\textrm{.}
\]
Young's theorem states that, at each point where $V$ is twice differentiable, $\frac{\partial^2 V}{\partial s \partial t} = \frac{\partial^2 V}{\partial t \partial s}$. Thus,
\begin{align*}
\frac{d}{ds} \Big|_{s=s_0} E(\gamma_s) &= \int_a^b \langle \frac{\partial^2 V}{\partial t \partial s}(s_0,t), \gamma_{s_0}'(t) \rangle \,dt\\
&= \int_a^b \frac{d}{d\tau} \Big|_{\tau = t} \langle \frac{\partial V}{\partial s}(s_0,\tau), \gamma_{s_0}'(\tau) \rangle \,dt - \int_a^b \langle \frac{\partial V}{\partial s}(s_0,t), \gamma_{s_0}''(t) \rangle \,dt\\
&= \langle W(s_0,t), \gamma_{s_0}'(t) \rangle \Big|_a^b - \int_a^b \langle W(s_0,t), \gamma_{s_0}''(t) \rangle \,dt\textrm{,}
\end{align*}
where the final equality follows from absolute continuity. \end{proof}

\noindent A \textbf{geodesic} in $M$ is a constant-speed path $\gamma : [a,b] \to M$ that is locally length-minimizing and, therefore, realizes the distance between nearby points. It is \textbf{minimal} if its length is the distance between its endpoints. Lytchak showed that, in embedded $C^{1,\alpha}$ submanifolds of smooth manifolds, geodesics are as regular as the submanifold itself.

\begin{theorem}[{\cite[Theorem 1.2]{Lytchak2005}}]\label{lytchak_theorem}
Let $0 < \alpha \leq 1$. If $M$ is an isometrically embedded $C^{1,\alpha}$ submanifold of a smooth Riemannian manifold, then geodesics in $M$ are $C^{1,\alpha}$.
\end{theorem}

\begin{remark}
Geodesics in manifolds with $C^{0,1}$ Riemannian metrics are $C^{1,1}$ \cite{LytchakYaman2006}, but without a sufficiently regular immersion into an ambient smooth manifold, many important properties are lost. For example, with $C^1$ metrics, the initial-value problem for the geodesic equation might fail to have unique solutions \cite{Hartman1950} and solutions to the geodesic equation might not realize the distance between nearby points \cite{HartmanWintner1951}. 
\end{remark}

\begin{lemma}\label{velocity normal to acceleration}
Let $M^n$ be a $C^{1,1}$ hypersurface in $\real^{n+1}$. If $\gamma : [a,b] \to M$ is any constant-speed $C^{1,1}$ path, then, for almost all $t \in [a,b]$, $\gamma''(t)$ exists and is orthogonal to $\gamma'(t)$.
\end{lemma}

\begin{proof} Since $|\gamma'|^2$ is constant and $\gamma'$ is Lipschitz, for almost all $t$, $\frac{d}{dt}|\gamma'|^2 = 2\langle \gamma''(t), \gamma'(t) \rangle = 0$. \end{proof}

\begin{lemma}\label{acceleration normal to hypersurface}
Let $M^n$ be a $C^{1,1}$ hypersurface in $\real^{n+1}$. If $\gamma : [a,b] \to M$ is any geodesic, then, for almost all $t \in [a,b]$, $\gamma''(t)$ exists and is orthogonal to $M$.
\end{lemma}

\begin{proof} By Theorem \ref{lytchak_theorem}, $\gamma$ is $C^{1,1}$. Suppose first that $\gamma$ is minimal. By Lemma \ref{velocity normal to acceleration}, $\langle \gamma'',\gamma' \rangle = 0$ almost everywhere. If $\gamma''$ is not orthogonal to $M$ on a set of positive measure, then one can pick a Lebesgue point $t_0$ of the intersection of that set and the set on which $\langle \gamma'',\gamma' \rangle = 0$. Let $w \neq 0$ be the projection of $\gamma''(t_0)$ onto $T_{\gamma(t_0)} M$. In $C^{1,1}$ graph coordinates $(x,f(x))$ around $p = \gamma(t_0)$, where $f : T_p M \to \real$, one has that $f(0) = 0$, $\nabla f(0) = 0$, and $\nabla f$ is Lipschitz continuous. Consequently, $\nabla^2 f$ exists almost everywhere and is essentially bounded. Following a standard ``rounding the corner'' argument, one may construct a piecewise $C^{1,1}$ variation $V_\varepsilon$ of $\gamma$ that pushes it in the direction of $w$ at time $t_0$ using a piecewise linear spike function with height $\varepsilon$ and support inside an interval of width $2\varepsilon$ centered at $t_0$. Bounding the two summands on the right-hand side of \eqref{fvf} separately on each half of $V_\varepsilon$, one may show that, for all sufficiently small $\varepsilon$ and almost all sufficiently small $s$, $\frac{d E(\gamma_{\varepsilon,s})}{ds} \leq \varepsilon^2 ( C \varepsilon - \frac{|w|^2}{2} )$, where $\gamma_{\varepsilon,s}(\cdot) = V_\varepsilon(s,\cdot)$ and $C$ is a constant that depends only on $|w|$, an upper bound for $|\gamma'|$, and a Lipschitz constant for $\nabla f$. Thus, $E(\gamma_{\varepsilon,s}) < E(\gamma)$ for all sufficiently small $\varepsilon$ and $s$. By Cauchy--Schwarz, $L(\gamma_{\varepsilon,s}) < L(\gamma)$, a contradiction. In the general case, $[a,b]$ can be partitioned into finitely many subintervals on which $\gamma$ is minimal; since the result holds on each subinterval, it holds overall. \end{proof}


\section{Convex hypersurfaces with long diametrical slices}
\label{slices_section}


Any $2$-plane containing a pair of antipodal points in a sphere of radius $r$ slices the sphere in a circle of length $2\pi r$. It follows from Crofton's formula (Theorem \ref{Crofton's formula} below) that this is the longest possible plane section of a convex hypersurface of extrinsic diameter $2r$. The primary aim of this section is to prove Theorem \ref{slicing property implies round}, which says that this property characterizes the sphere. 


\subsection{Width and Diametrical Slices} 

In this subsection, we introduce several definitions, survey elementary results from convex geometry, and prove Lemma \ref{slices are closed geodesics}, which describes the geodesics in $C^{1,1}$ convex hypersurfaces satisfying the hypotheses of Theorem \ref{slicing property implies round}. The results developed here are used in the proof of Theorem \ref{slicing property implies round} and, in turn, Theorems \ref{mean_curvature_theorem} and \ref{mean_width_theorem}.

\begin{definition}\label{width_defn}
Let $M^n$ be a closed, convex hypersurface in $\real^{n+1}$. For each $u \in S^n$, the \textbf{width of $M$ in the direction of $u$}, denoted $w(M,u)$, is the length of the orthogonal projection of $M$ onto the line spanned by $u$. Equivalently, it is the distance between the two supporting hyperplanes to $M$ orthogonal to $u$. The \textbf{maximum width of $M$} is $w^{+}(M) = \max_{u \in S^n} w(M,u)$, the \textbf{minimum width of $M$} is $w^{-}(M) = \min_{u \in S^n} w(M,u)$, and the \textbf{mean width of $M$} is
\begin{align}
\displaystyle \Xi(M) = \frac{1}{\sigma_n} \int_{S^n} w(M,u)\,du. 
\end{align}
\end{definition}

\noindent It follows immediately from the definition that the width function $w(M,\cdot) : S^n \to [0,\infty)$ is even, i.e., $w(M,-u) = w(M,u)$. If $\Pi$ is an affine subspace of $\real^{n+1}$ that contains a vector parallel to $u \in S^n$, and if $\Pi \cap M \neq \emptyset$, then 
\begin{equation}\label{slice_width_inequality}
\displaystyle w(\Pi \cap M, u) \leq w(M, u). 
\end{equation}

\noindent For planar convex curves, Crofton's formula states that mean width is equivalent to length. 

\begin{theorem}[Crofton's formula \cite{Schneider1993}]\label{Crofton's formula}
Let $\kappa$ be a closed, convex, planar curve. Then, 
\[
L(\kappa) = \pi \Xi(\kappa). 
\]
\end{theorem}

\noindent This implies that convex curves of constant width $w$ have length $\pi w$, a result also known as Barbier's theorem. The following inequality for the lengths of nested convex curves is also an immediate corollary of Crofton's formula.

\begin{corollary}\label{crofton_corollary}
Let $\kappa_1$ and $\kappa_2$ be closed, convex, planar curves, and suppose $\kappa_1$ lies in the region bounded by $\kappa_2$. Then $L(\kappa_1) \leq L(\kappa_2)$, with equality if and only if $\kappa_1 = \kappa_2$.
\end{corollary}

\noindent Crofton's formula also gives the following formula for mean width.

\begin{lemma}\label{integral_geometric_formula}
Let $M^{n}$ be a closed, convex hypersurface in $\real^{n+1}$, $\mathrm{Gr}(2,n+1)$ the Grassmannian of unoriented $2$-planes in $\real^{n+1}$, and $d\Pi$ the $SO(n+1)$-invariant measure on $\mathrm{Gr}(2,n+1)$. For each $\Pi \in \mathrm{Gr}(2,n+1)$, let $\phi_{\Pi}$ denote orthogonal projection onto $\Pi$. Then, 
\begin{align}\label{int_geom_form_eqn}
\displaystyle \Xi(M) = \int_{\mathrm{Gr}(2,n+1)} \Xi(\partial \phi_{\Pi}(M)) \ d\Pi = \frac{1}{\pi} \int_{\mathrm{Gr}(2,n+1)} L\left(\partial \phi_{\Pi}(M)\right) \ d\Pi. 
\end{align}
\end{lemma}

\begin{proof} The first equality in \eqref{int_geom_form_eqn} follows from the definition of $\Xi$ via an elementary integral-geometric calculation and the second from Theorem \ref{Crofton's formula}. \end{proof}

\begin{definition}
Let $M^n$ be a closed, convex hypersurface in $\real^{n+1}$ and $S^{n} \subset \real^{n+1}$ the unit sphere. The \textbf{support function of $M$} is the function $h_{M}: S^{n} \to \real$ defined by
\begin{equation}
\displaystyle h_{M}(u) = \max\limits_{x \in M} \langle x, u \rangle.
\end{equation}
Equivalently, $h_{M}(u)$ is the larger of the signed distances from the origin to the two supporting hyperplanes of $M$ orthogonal to $u$. If $\Pi$ is an affine subspace of $\real^{n+1}$, the restriction of $h_{M}$ to the unit sphere in the linear subspace parallel to $\Pi$ is denoted $h_{M}|_\Pi$.
\end{definition}

\noindent The width in each direction may be computed using the support function: for each $u \in S^n$, 
\begin{equation}\label{width_support}
\displaystyle w(M,u) = h_{M}(u) + h_{M}(-u). 
\end{equation}
Note also that, for any convex hypersurface $M^{n}$ in $\real^{n+1}$ and any affine subspace $\Pi$ with $\Pi \cap M \neq \emptyset$,
\begin{equation}\label{slice_support}
\displaystyle h_{M}|_\Pi \geq h_{\Pi \cap M}. 
\end{equation}

\noindent It is well known \cite[Lemma 1.8.10]{Schneider1993} that $h_M$ is Lipschitz continuous, with a Lipschitz constant given by $\max_{x \in M} |x|$. Since $w(M,u) = h_M(u) + h_M(-u)$ for each $u \in S^n$, a Lipschitz constant for $w(M,\cdot)$ is given by $2\max_{x \in M} |x|$.

\begin{definition}\label{chebyshev_radius}
Let $M^n$ be a closed, convex hypersurface in $\real^{n+1}$. The \textbf{Chebyshev radius of $M$}, denoted $R_c(M)$, is the smallest radius of a closed ball in $\real^{n+1}$ that contains $M$.
\end{definition}

\noindent There is always a unique minimal enclosing ball of $M$ of radius $R_c(M)$, and its center must lie in the region bounded by $M$. By translating $M$ to place the center of its minimal enclosing ball at the origin, one finds that $w(M,\cdot)$ has Lipschitz constant $2R_c (M)$. If $M$ has maximum width $w^{+}(M)$ and a $2$-plane $\Pi$ meets $M$ in a curve of length $\pi w^{+}(M)$, then \eqref{slice_width_inequality} and Theorem \ref{Crofton's formula} imply that $\Pi \cap M$ has constant width $w^{+}(M)$ and, for all $u$ in the unit circle in $\Pi$,
\[
w(\Pi \cap M,u) = w^{+}(M) = h_{\Pi \cap M}(u) + h_{\Pi \cap M}(-u) \leq h_{M}(u) + h_{M}(-u) \leq w^{+}(M) \textrm{.}
\]
By \eqref{slice_support}, this implies that the even parts of $h_{\Pi \cap M}$ and $h_{M}|_\Pi$ agree, and, consequently, the odd parts do as well. Thus, $h_{\Pi \cap M} = h_{M}|_\Pi$. The next result is a quantitative refinement of this fact.

\begin{lemma}\label{width stability}
For each pair of positive numbers $R$ and $\varepsilon$, there is $\delta > 0$ such that, if $M^n$ is a closed, convex hypersurface in $\real^{n+1}$ of maximum width $w^{+}(M) \leq R$, and if $\Pi$ is a $2$-plane in $\real^{n+1}$ that intersects $M$ in a curve of length $L(\Pi \cap M) > \pi w^{+}(M) - \delta$, then $w^{-}(\Pi \cap M) > w^{+}(M) - \varepsilon$. In particular, if $M$ is a closed, convex hypersurface of maximum width $w^{+}(M)$ and $\Pi$ is a $2$-plane that intersects $M$ in a curve of length $\pi w^{+}(M)$, then $\Pi \cap M$ has constant width $w^{+}(M)$ and $h_{\Pi \cap M} = h_{M}|_\Pi$.  
\end{lemma}

\begin{proof} Without loss of generality, suppose $0 < \varepsilon < 2\pi R$, and set $\delta = \frac{\varepsilon^2}{4R}$. Let $M^n$ be a closed, convex hypersurface in $\real^{n+1}$ with $w^{+}(M) \leq R$, and let $\Pi$ be a $2$-plane such that $L(\Pi \cap M) > \pi w^{+}(M) - \delta$. Assume, toward a contradiction, that $w^{-}(\Pi \cap M) \leq w^{+}(M) - \varepsilon$. Let $\theta_0$ be a direction in which $w(\Pi \cap M,\theta_0) = w^{-}(\Pi \cap M)$. Since $R_c(\Pi \cap M) \leq R_c(M) \leq R$, $2R$ is a Lipschitz constant for $w(\Pi \cap M,\cdot)$. Therefore, on the interval $I = [\theta_0 - \frac{\varepsilon}{2R},\theta_0 + \frac{\varepsilon}{2R}]$ (where $S^1$ is identified with $\real/2\pi \integer$), $w(\Pi \cap M,\theta) \leq w^{-}(\Pi \cap M) + 2R|\theta - \theta_0|$. Integrating over $I$, one obtains
\[
\int_I w(\Pi \cap M,\theta)\,d\theta \leq \frac{\varepsilon w^{-}(\Pi \cap M)}{R} + \frac{\varepsilon^2}{2R}.
\]
At the same time,
\[
\int_{S^1 \setminus I} w(\Pi \cap M,\theta)\,d\theta \leq 2 \Big( \pi - \frac{\varepsilon}{2R} \Big)w^{+}(M).
\]
By Theorem \ref{Crofton's formula},
\begin{align*}
&\pi w^+(M) - \delta < L(\Pi \cap M) = \frac{1}{2} \int_{S^1} w(\Pi \cap M,\theta)\,d\theta\\
&\leq \frac{\varepsilon w^{-}(\Pi \cap M)}{2R} + \frac{\varepsilon^2}{4R} + \Big( \pi - \frac{\varepsilon}{2R} \Big) w^{+}(M).
\end{align*}
It follows that $\delta > \frac{\varepsilon^2}{4R}$, a contradiction. \end{proof}

\begin{definition}\label{long_diametrical_slices}
A \textbf{diametrical segment} of a closed, convex hypersurface $M^n$ in $\real^{n+1}$ is a chord of $M$ of length $w^{+}(M)$. Two points of $M$ are \textbf{diametrically opposed} if they are connected by a diametrical segment. If a $2$-plane $\Pi$ contains a diametrical segment of $M$, its intersection with $M$ is a \textbf{diametrical slice}. If all diametrical slices of $M$ have length $\pi w^{+}(M)$, then $M$ has \textbf{long diametrical slices}. 
\end{definition}

\noindent Note that every planar convex curve has long diametrical slices, according to the definition above, and that the conclusion of Theorem \ref{slicing property implies round} is not valid in this case.

\begin{lemma}\label{width_equals_diameter}
Let $M^n$ be a closed, convex hypersurface in $\real^{n+1}$. Then, the following hold:
\begin{itemize}
\item[\textbf{(a)}] If $x,y \in M$ realize $\max_{x,y \in M} |y - x|$, then the hyperplanes orthogonal to $y - x$ at $x$ and $y$ are supporting hyperplanes for $M$. 
\item[\textbf{(b)}] $w^{+}(M) = \max\limits_{x,y \in M} |y - x|$. In particular, $x$ and $y$ are diametrically opposed if and only if $|y - x| = w^{+}(M)$. 
\item[\textbf{(c)}] $M$ has at most one diametrical segment parallel to each $u \in S^n$. 
\item[\textbf{(d)}] The following are equivalent: 
\begin{itemize}
\item[(i)]  $M$ has constant width; 
\item[(ii)] For each $u \in S^n$, $M$ has a diametrical segment parallel to $u$;  
\item[(iii)] Each point of $M$ is diametrically opposed to a (not necessarily unique) point of $M$. 
\end{itemize}
\end{itemize}
\end{lemma}

\begin{proof} The result in (a) and the fact that $w^{+}(M) \leq \max_{x,y \in M} |y - x|$ are elementary. If $x$ and $y$ are points whose distance realizes $w^{+}(M)$, then the hyperplanes through them and orthogonal to the line connecting them are support hyperplanes for $M$, so $w^{+}(M) \geq |y - x|$, giving the equality in (b). If $M$ had a pair of parallel diametrical segments, one of the diagonals in the parallelogram they determined would be longer than them, contradicting (b) and proving (c). Part (d) follows from (b) and (c). \end{proof}

\noindent Part (b) of Lemma \ref{width_equals_diameter} implies that every convex hypersurface has a diametrical segment. The well-known fact that hypersurfaces of constant width are \textbf{strictly convex}, in the sense that they contain no line segments, can also be derived from Lemma \ref{width_equals_diameter}. 

\begin{definition}\label{normal_cycle_defn}
Let $M^n$ be a closed, convex hypersurface in $\real^{n+1}$ bounding a domain $\mathcal{K}$. We denote by $N(M)$ the \textbf{normal cycle of $\mathcal{K}$}, that is, the set of unit tangent vectors to $\real^{n+1}$ that point out of $\mathcal{K}$ and whose orthogonal complements are supporting hyperplanes to $M$.
\end{definition}

\noindent When $M$ is $C^{1}$, $N(M)$ is simply the outer component of the unit normal bundle to $M$. In general, $N(M)$ is a Lipschitz submanifold of the unit tangent bundle of $\real^{n+1}$. We refer to Z\"ahle \cite{Zahle1987} for background about the normal cycle and its fundamental properties. Below, the notation $v_x$ will sometimes be used to emphasize the basepoint $x$ of a vector $v \in T_x \real^{n+1}$. We will also sometimes identify a vector $v$ based at $x$ with the line segment connecting $x$ to $x + v$.

\begin{lemma}\label{normal vectors agree}
Let $M^n$ be a closed, convex hypersurface in $\real^{n+1}$ and $\Pi$ an affine subspace of $\real^{n+1}$ with $\Pi \cap M \neq \emptyset$. Denote by $\phi_{\Pi}$ the orthogonal projection onto $\Pi$. Then, the following are equivalent:
\begin{itemize}
\item[\textbf{(a)}] $h_{\Pi \cap M} = h_{M}|_\Pi$,
\item[\textbf{(b)}] $N(\Pi \cap M) = N(M) \cap T_\Pi$, 
\item[\textbf{(c)}] $\Pi \cap M = \partial \phi_{\Pi} (M)$.
\end{itemize}
\end{lemma}

\begin{proof} Note that $h_{M}|_\Pi$ agrees with the support function of $\partial \phi_{\Pi}(M)$ by the following argument: any supporting hyperplane to $\partial \phi_{\Pi}(M)$ extends orthogonally to a supporting hyperplane to $M$, though not necessarily based at the same point; conversely, any supporting hyperplane to $M$ orthogonal to a vector in $\Pi$ projects to a supporting hyperplane to $\partial \phi_{\Pi}(M)$. Suppose $h_{\Pi \cap M} = h_{M}|_\Pi = h_{\partial \phi_{\Pi}(M)}$. Then, $N(\Pi \cap M) = N(\partial \phi_{\Pi}(M))$. Note that $N(M) \cap T_\Pi \real^{n+1} \subseteq N(\Pi \cap M)$ by definition. If $u_x \in N(\Pi \cap M)$, then $u_x^\perp \subset T_x \real^{n+1}$ is a supporting hyperplane to $M$, so $N(\Pi \cap M) \subseteq N(M) \cap T_\Pi \real^{n+1}$. Thus, (a) $\implies$ (b). It is clear that (b) $\implies$ (c) $\implies$ (a). \end{proof}

\begin{lemma}\label{constant width}
If $M^n$ is a closed, convex hypersurface in $\real^{n+1}$ with long diametrical slices, then $M$ has constant width.
\end{lemma}

\begin{proof} Let $u \in \real^{n+1}$ be any unit vector, and consider any $2$-plane $\Pi$ containing both a diametrical segment of $M$ and a vector parallel to $u$. By assumption, $\Pi \cap M$ has length $\pi w^{+}(M)$, so Lemma \ref{width stability} implies it has constant width $w^{+}(M)$. It therefore contains a diametrical segment parallel to $u$, which must also be a diametrical segment of $M$. Thus, $M$ has width $w^{+}(M)$ in the direction of $u$. \end{proof}

\begin{lemma}\label{normals are diameters in constant-width curves}
If $\kappa$ is a convex curve in $\real^2$ of constant width $w$, then for each $u_x \in N(\kappa)$, $-w u_x$ is a diametrical segment of $\kappa$.
\end{lemma}

\begin{proof} Let $u_x \in N(\kappa)$. By strict convexity, $x$ is the unique point in $\kappa$ at which $u_x^\perp$ is a supporting line to $\kappa$ and $u_x$ points away from the domain bounded by $\kappa$. At the same time, by strict convexity, there is a unique point $y$ at which $-w u_x$ is a diametrical segment of $\kappa$. Since $u_x$ points away from $\kappa$ at $y$ and $u_x^\perp$ is a supporting line to $\kappa$ at $y$, $y$ must equal $x$. \end{proof}

\begin{lemma}\label{normals are diameters in constant-width hypersurfaces}
Let $M^n$ be a closed, convex hypersurface in $\real^{n+1}$ with long diametrical slices and constant width $w$. Then, for each $u_x \in N(M)$, $-w u_x$ is a diametrical segment of $M$. In particular, if $M$ is $C^{1}$, then each $x$ in $M$ is diametrically opposed to a unique point, which we denote by $\alpha(x)$. 
\end{lemma}

\begin{proof} Let $u_x \in N(M)$. Since $M$ has constant width, there exists some diametrical segment $\ell$ of $M$ parallel to $-u$. Let $\Pi$ be any $2$-plane containing $\ell$ and $x$. By Lemmas \ref{width stability} and \ref{normal vectors agree}, $\Pi \cap M$ has constant width $w$ and $N(\Pi \cap M) = N(M) \cap T_\Pi \real^{n+1}$. By Lemma \ref{normals are diameters in constant-width curves}, for each $v_y \in N(\Pi \cap M)$, $-w v_y$ is a diametrical segment of $\Pi \cap M$ and, consequently, of $M$ itself. Since $u_x \in N(M) \cap T_\Pi \real^{n+1}$, $-w u_x$ is a diametrical segment of $M$. \end{proof}

\noindent The proof of Lemma \ref{slices are closed geodesics} draws on the following differentiability result for Lipschitz functions. 

\begin{lemma}\label{bypass chain rule}
Let $F : U \subseteq \real^n \to \real^m$ be Lipschitz continuous. Suppose $f,g : (a,b) \to U$ are differentiable at $t \in (a,b)$, $f(t) = g(t)$, and $f'(t) = g'(t)$. If $F \circ f$ is differentiable at $t$, then $F \circ g$ is differentiable at $t$, and $(F \circ g)'(t) = (F \circ f)'(t)$.
\end{lemma}

\begin{proof} Let $C$ be a Lipschitz constant for $F$. Then,
\begin{align*}
&\lim_{h \to 0} \Big| \frac{(F \circ g)(t + h) - (F \circ f)(t + h)}{h} \Big| \leq C\lim_{h \to 0} \Big| \frac{g(t + h) - f(t + h)}{h} \Big|\\
= \ &C\lim_{h \to 0} \Big| \frac{g(t + h) - g(t)}{h} - \frac{f(t + h) - f(t)}{h} \Big| = C|g'(t) - f'(t)| = 0.
\end{align*}
It follows that $\lim\limits_{h \to 0} \frac{(F \circ g)(t + h) - (F \circ g)(t)}{h} = \lim\limits_{h \to 0} \frac{(F \circ f)(t + h) - (F \circ f)(t)}{h} = (F \circ f)'(t)$, which implies that $(F \circ g)'(t) = (F \circ f)'(t)$. \end{proof}

\begin{lemma}\label{slices are closed geodesics}
If $M^n$ is a closed, convex, $C^{1,1}$ hypersurface in $\real^{n+1}$ with long diametrical slices and constant width $w$, then the geodesics of $M$ coincide with its diametrical slices, and the intrinsic diameter and injectivity radius of $M$ are both $\frac{\pi w}{2}$. 
\end{lemma}

\begin{proof}
Let $\gamma : (-\varepsilon, \varepsilon) \to M$ be a unit-speed parametrization of a geodesic segment. By Theorem \ref{lytchak_theorem}, $\gamma$ is $C^{1,1}$. For each $t \in (-\varepsilon, \varepsilon)$, let $\Pi_{\gamma(t)}$ be the $2$-plane through $\gamma(t)$ and parallel to the $2$-plane spanned by $\gamma'(t)$ and $N(\gamma(t))$. Lemma \ref{normals are diameters in constant-width hypersurfaces} implies that $\Pi_{\gamma(t)} \cap M$ is a diametrical slice. Lemma \ref{acceleration normal to hypersurface} implies that, for a bounded and measurable $c : (-\varepsilon,\varepsilon) \to \real$, $\gamma'' = c(t)N(\gamma(t))$ almost everywhere. Because $\gamma$ and $N$ are Lipschitz continuous, $N \circ \gamma : (-\varepsilon, \varepsilon) \to \real^{n+1}$ is Lipschitz continuous and, consequently, differentiable for almost all $t$. For each such $t$, let $\sigma : (-\varepsilon, \varepsilon) \to M$ be a unit-speed parametrization of $\Pi_{\gamma(t)} \cap M$ with $\sigma(t) = \gamma(t)$ and $\sigma'(t) = \gamma'(t)$. By Lemma \ref{bypass chain rule}, the differentiability of $N \circ \gamma$ at $t$ implies that $(N \circ \sigma)'(t) = (N \circ \gamma)'(t)$. Lemma \ref{width stability} implies that $N$ restricts to the unit normal vector field along $\Pi_{\gamma(t)} \cap M$ and, therefore, that $(N \circ \sigma)'(t)$ is a scalar multiple of $\sigma'(t)$. Together, these imply that $(N \circ \gamma)'(t) = b(t) \gamma'(t)$ almost everywhere for a bounded and measurable $b : (-\varepsilon, \varepsilon) \to \real$. Define $F : (-\varepsilon, \varepsilon) \to \Lambda^{2}\left(\real^{n+1}\right)$ by $F(t) = \gamma'(t) \wedge (N \circ \gamma)(t)$. Because $\gamma' : (-\varepsilon, \varepsilon) \to \real^{n+1}$ and $N \circ \gamma : (-\varepsilon, \varepsilon) \to \real^{n+1}$ are Lipschitz continuous, $F$ is as well. Its image is a rectifiable curve in the vector space $\Lambda^{2}\left(\real^{n+1}\right)$, and its length is $\int_{-\varepsilon}^{\varepsilon}|F'(t)|dt$. By the above,
\begin{align*}
F'(t) &= \gamma''(t) \wedge (N \circ \gamma)(t) + \gamma'(t) \wedge (N \circ \gamma)'(t)\\
&= c(t)(N \circ \gamma)(t) \wedge (N \circ \gamma)(t) + b(t) \gamma'(t) \wedge \gamma'(t)\\
&= 0
\end{align*}
almost everywhere. The map $F$ is therefore constant; thus, for all $t \in (-\varepsilon, \varepsilon)$, $\gamma'(t)$ is contained in the $2$-plane spanned by $\gamma'(0)$ and $(N \circ \gamma)(0)$. Because $\gamma(t) - \gamma(0) = \int_{0}^{t} \gamma'(\tau) d \tau$, this implies that $\gamma : (-\varepsilon, \varepsilon) \to M$ remains in the diametrical slice $\Pi_{\gamma(0)} \cap M$. The connectedness of geodesics and diametrical segments therefore implies that they define the same family of curves in $M$. 

Given $x \in M$, let $\alpha(x)$ be the unique point diametrically opposed to $x$, as in Lemma \ref{normals are diameters in constant-width hypersurfaces}. By the characterization of geodesics above, the injective geodesic segments from $x$ to $\alpha(x)$ are given by arcs of intersections of $M$ with $2$-planes containing $x$ and $\alpha(x)$. Given two such geodesics $\gamma_{1}$ and $\gamma_{2}$ that are arcs of $\Pi_{1} \cap M$ and $\Pi_{2} \cap M$, respectively, let $\hat{\Pi}$ be a $3$-dimensional affine space containing $\Pi_{1}$ and $\Pi_{2}$, and let $\Pi_{0}$ be the $2$-plane in $\hat{\Pi}$ through $\frac{x+\alpha(x)}{2}$ and orthogonal to the diametrical segment $-w N(x)$ between $x$ and $\alpha(x)$.  For each real number $\vartheta$, let $\Psi_{\vartheta}:M \to M$ be the map given by projecting $M$ radially inward onto a sphere centered at $\frac{x+\alpha(x)}{2}$, rotating this sphere by an angle $\vartheta$ in the plane $\Pi_{0}$ according to a fixed orientation, and then radially projecting back out onto $M$. This family of mappings, for $0 \leq \vartheta \leq \angle(\Pi_{1},\Pi_{2})$, transforms $\gamma_{1}$ into $\gamma_{2}$, and Lemma \ref{first variation formula} implies that all segments in this variation have the same length. As a result, all geodesic segments linking $x$ and $\alpha(x)$ have length $\frac{\pi w}{2}$, which must therefore be their intrinsic distance. Geodesics in $M$ must therefore minimize between any pair of diametrically opposed points and cannot minimize past them.
\end{proof}

\begin{remark}\label{blaschke_remark}
If $M$ is smooth, the fact that all geodesics beginning at a point $x \in M$ meet at $\alpha(x)$ implies the Riemannian metric on $M$ is a Blaschke metric, cf. \cite{Besse1978}. The identity $(N \circ \sigma)'(t) = b(t)\sigma(t)$ appearing in the proof of Lemma \ref{slices are closed geodesics} implies that all tangent vectors to $M$ are principal vectors and, therefore, that $M$ is a totally umbilic hypersurface. Both of these conditions are known to imply that $M$ is isometric to a round sphere in the smooth case, but, to the best of the authors' knowledge, it has not been shown that these characterizations hold under the conditions in Lemma \ref{slices are closed geodesics}. 
\end{remark}


\subsection{Zindler Curves of Constant Width}

Let $\kappa$ be a closed curve in the plane $\real^2$. A chord of $\kappa$ is \textbf{diametral} if its length is equal to the width of $\kappa$ in the direction it determines. Klamkin asked whether every closed planar curve of constant width whose diametral chords all bisect its perimeter must be a circle (see \cite{HammerSmith1964} for a discussion). This was answered in the affirmative by Besicovitch \cite{Besicovitch1961} and Smith \cite{Smith1961}.

\begin{theorem}[Besicovitch--Smith]\label{besicovitch_smith_thm}
Every closed planar curve of constant width whose diametral chords all bisect its perimeter is a circle.
\end{theorem}

\noindent A closed planar curve with the property that all chords that bisect its perimeter have the same length is called \textbf{Zindler}, as the first nontrivial examples were constructed in \cite{Zindler1921}. This is equivalent to all chords that bisect its area having the same length. 

\begin{lemma}\label{zindler_lemma}
For a closed planar curve $\kappa$ of constant width $w$, the following are equivalent:	
\begin{itemize}
\item[\textbf{(i)}] $\kappa$ is Zindler;
\item[\textbf{(ii)}] Every diametral chord of $\kappa$ bisects its perimeter;
\item[\textbf{(iii)}] Every diametral chord of $\kappa$ bisects its area.
\end{itemize}
\end{lemma}

\begin{proof} Define a map $\gamma : S^1 \to \kappa$ that takes $\theta$ to the unique point on $\kappa$ at which $\theta$ is an outward-pointing normal vector. Identify $S^1$ with $\real / 2\pi\integer$.  Then, $\gamma(\theta)$ and $\gamma(\theta + \pi)$ are always connected by a diametral chord. Suppose (i) holds. Since $L(\gamma|_{[\theta,\theta + \pi]})$ is a continuous function of $\theta$ and $L(\gamma|_{[\theta,\theta + \pi]}) + L(\gamma|_{[\theta + \pi,\theta + 2\pi]}) = \pi w$, the intermediate value theorem implies that some diametral chord bisects the perimeter. By the Zindler condition, all such bisecting chords have length $w$, so they are all diametral. Since diametral chords are unique in each direction, this proves (ii). A similar argument shows that (i) implies (iii). Since $\kappa$ has constant width, it is clear that (ii) implies (i) and that (iii) implies (i). \end{proof}

\noindent By Lemma \ref{zindler_lemma}, Theorem \ref{besicovitch_smith_thm} may be restated as follows. Under additional regularity assumptions, this was proved earlier by Geppert \cite{Geppert1940}.

\begin{corollary}\label{zindler_corollary}
Every Zindler curve of constant width is a circle. 
\end{corollary}




\begin{proof}[{\em 3.3.} \textbf{\em Proof of Theorem \ref{slicing property implies round}}]
Let $M^n$ be a closed, convex hypersurface in $\real^{n+1}$ with long diametrical slices. By Lemma \ref{constant width}, $M$ has constant width $w = w^{+}(M)$. Suppose first that $M$ is $C^{1,1}$, and define $\alpha:M \to M$ to be the map taking each $x \in M$ to its unique diametrically opposed point, as in Lemma \ref{normals are diameters in constant-width hypersurfaces}. That is, $\alpha(x) = x - w N(x)$. Lemma \ref{slices are closed geodesics} implies that each $2$-plane $\Pi$ containing $x$ and $\alpha(x)$ slices $M$ in a closed geodesic, both of whose segments minimize between $y$ and $\alpha(y)$ for all $y \in \Pi \cap M$. Each diametrical slice is therefore a Zindler curve of constant width $w$ and, by Corollary \ref{zindler_corollary}, must be a circle of radius $\frac{w}{2}$. Any such circle containing $x \in M$ is centered at $\hat{x} = x - \frac{w}{2} N(x)$. As each $y \in M$ belongs to some $2$-plane containing a diametrical slice through $x$, the point $\hat{x}$ coincides with the point $\hat{y}$ corresponding to $y$. All such circles are therefore concentric, $\alpha$ is a reflection through $\hat{x}$, and $M$ is centrally symmetric. Since its support function must be of the form $f(x) = \frac{w}{2} + \langle x, \hat{x} \rangle$, $M$ is a round sphere.
	
In the general case, for each $\varepsilon > 0$, let $M_\varepsilon$ be the outer parallel hypersurface to $M$ at distance $\varepsilon$, i.e., $M_\varepsilon = \{ x + \varepsilon u_x \mid u_x \in N(M) \}$. Then, $M_\varepsilon$ is $C^{1,1}$ \cite[Theorem 4.8, Corollary 4.9]{Federer1959} and, because $M$ has constant width and the support function $h_{M_{\varepsilon}}$ of $M_{\varepsilon}$ is given by $h_{M} + \varepsilon$, has constant width equal to $w + 2\varepsilon$. Lemmas \ref{width stability}, \ref{normal vectors agree}, and \ref{normals are diameters in constant-width hypersurfaces} imply that, for each $2$-plane $\Pi$ that meets $M$ in a diametrical slice, $N(\Pi \cap M) = N(M) \cap T_\Pi \real^{n+1}$; therefore, $\Pi \cap M_\varepsilon$ agrees with the outward normal expansion of $\Pi \cap M$ by $\varepsilon$ in $\Pi$. So, $\Pi \cap M_{\varepsilon}$ is a diametrical slice of $M_{\varepsilon}$ with constant width $w + 2\varepsilon$. By Barbier's theorem, $L(\Pi \cap M_\varepsilon) = \pi(w + 2\varepsilon)$. Therefore, $M_{\varepsilon}$ has long diametrical slices and, by the argument above, is a round sphere. Because all $M_\varepsilon$ are round spheres, $M$ is as well. \end{proof}


\section{Mean width and the Birkhoff invariant}
\label{mean_width_section}


In this section, we prove Theorem \ref{mean_width_theorem}. If $M$ is a topological manifold, denote by $\Omega(M)$ and $\Omega^{0}(M)$ its free loop space and the space of constant loops in $M$, respectively, and let $B^k$ denote the closed unit ball in $\real^k$ and $S^{k-1}$ its boundary. The following definition of a sweepout is drawn from \cite{ChengAlshawa2025}.

\begin{definition}\label{birkhoff invariant}
Let $M$ be a topological manifold endowed with a length metric. A \textbf{sweepout} of $M$ is a map $\Phi : B^k \to \Omega(M)$ that takes $S^{k-1}$ into $\Omega^{0}(M)$ and represents a nontrivial element of the relative homotopy group $\pi_k(\Omega(M), \Omega^{0}(M))$, where $k \geq 1$ is the smallest integer such that $\pi_k(\Omega(M), \Omega^{0}(M))$ is nontrivial. Its \textbf{length} is $L(\Phi) = \max \,\{ L(\Phi(x)) \mid x \in B^k \}$. The \textbf{Birkhoff invariant} of $M$ is $\beta(M) = \inf \,\{ L(\Phi) \mid \Phi \textrm{ is a sweepout of } M \}$.
\end{definition}


\subsection{Hausdorff Distance and the Birkhoff Invariant}

Recall that the \textbf{Hausdorff distance} between nonempty compact subsets $X$ and $Y$ of $\real^{n+1}$ is
\[
d_H(X,Y) = \max \big\{ \displaystyle \sup_{x \in X} d(x,Y), \displaystyle \sup_{y \in Y} d(y,X) \big\}.
\]
This satisfies the axioms of a metric on the set of nonempty, compact subsets of $\real^{n+1}$ and defines a complete metric, the \textbf{Hausdorff metric}.

\begin{lemma}\label{width_continuity}
For closed, convex hypersurfaces $M,N$ in $\real^{n+1}$ and $u \in S^{n}$, the following hold:
\begin{itemize}
\item[\textbf{(a)}] $|h_{M}(u) - h_{N}(u)| \leq d_H(M,N)$
\item[\textbf{(b)}] $|w(M,u) - w(N,u)| \leq 2d_H(M,N)$
\end{itemize}
In particular, mean width is $2$-Lipschitz continuous with respect to the Hausdorff metric, and, for planar convex curves, length is $2\pi$-Lipschitz continuous, with the convention that the length of a line segment is counted twice if it occurs as a limit of curves bounding open sets. 
\end{lemma} 

\begin{proof} If $x \in M$ were a point with $\langle x,u \rangle > h_{N}(u) + d_H(M,N)$, then, for all $x' \in N$, $|x - x'| \geq |\langle x - x', u \rangle | = \langle x,u \rangle - \langle x',u \rangle > d_H(M,N)$, contradicting the definition of $d_H(M,N)$. Reversing the roles of $M$ and $N$ implies (a), and (b) follows from (a) and the identity $w(M,u) = h_{M}(u) + h_{M}(-u)$. The Lipschitz continuity of mean width follows immediately from its definition, and the Lipschitz continuity of the length of plane curves follows from Theorem \ref{Crofton's formula}. \end{proof}

\begin{lemma}\label{radius_continuity} 
For a closed, convex hypersurface $M$ in $\real^{n+1}$, let $\rho(M)$ be its \textbf{inradius}, that is, the largest $r \geq 0$ such that the domain $\mathcal{K}$ bounded by $M$ contains a ball of radius $r$. Then, $\rho$ and the Chebyshev radius $R_{c}$ are both $1$-Lipschitz continuous with respect to the Hausdorff metric. 
\end{lemma}

\noindent Note that, unlike the closed ball of radius $R_{c}(M)$ bounding a convex hypersurface $M$, a ball of radius $\rho(M)$ in the domain bounded by $M$ might not be unique. 

\begin{proof}[Proof of Lemma \ref{radius_continuity}] To establish the result for the inradius, suppose that a ball $B(x,\rho(M))$ is contained within the region bounded by $M$. Since the outward normal expansion of $N$ by $d_H(M,N)$ (i.e., the Minkowski sum of the region bounded by $N$ and a ball of radius $d_H(M,N)$) contains $M$, it also contains $B(x,\rho(M))$. If $\rho(M) > d_H(M,N)$, then $x$ must be contained within the region bounded by $N$. It follows that $B(x, \rho(M) - d_H(M,N))$ is contained within the region bounded by $N$, so $\rho(N) \geq \rho(M) - d_H(M,N)$. This inequality also clearly holds if $\rho(M) \leq d_H(M,N)$. In the same way, $\rho(M) \geq \rho(N) - d_H(M,N)$. So, $|\rho(N) - \rho(M)| \leq d_H(M,N)$. To establish the result for the Chebyshev radius, let $q$ be the center of the unique closed ball of radius $R_{c}(M)$ containing $M$. Because $N$ is contained in the $d_H(M,N)$-neighborhood of $M$, it is contained in the ball of radius $R_{c}(M) + d_H(M,N)$ about $q$, which implies that $R_{c}(N) \leq R_{c}(M) + d_H(M,N)$. Similarly, $R_{c}(M) \leq R_{c}(N) + d_H(M,N)$. So, $|R_{c}(N) - R_{c}(M)| \leq d_H(M,N)$.
\end{proof}

\begin{lemma}\label{birkhoff inradius bound}
Suppose $M^n$ and $N^n$ are closed, convex hypersurfaces in $\real^{n+1}$. If $d_H(M,N) < \varepsilon < \min \{ \rho(M), \rho(N) \}$, then $|\beta(M) - \beta(N)| < \varepsilon \cdot \max \big\{ \frac{\beta(M)}{\rho(M)}, \frac{\beta(N)}{\rho(N)} \big\}$.
\end{lemma}

\begin{proof} Let $p$ be any point in the interior of the region $\mathcal{K}$ bounded by $M$ such that the closed ball of radius $\rho(M)$ around $p$ is contained within $\mathcal{K}$. Without loss of generality, translate $M$ and $N$ so that $p$ is the origin. For $\eta = \frac{\rho(M)}{\rho(M) + \varepsilon}$, the closed, convex hypersurface
\[
N_\eta = \eta N = \{ \eta x \,|\, x \in N \}
\]
lies entirely in the interior of $\mathcal{K}$. The projection onto $N_\eta$ from its exterior is $1$-Lipschitz, so, if $\Phi$ is any sweepout of $M$, projecting it onto $N_\eta$ produces a sweepout $\Phi_\eta$ of $N_\eta$ such that $L(\Phi_\eta) \leq L(\Phi)$. The map from $N_\eta$ to $N$ that scales by $\frac{1}{\eta}$ has Lipschitz constant $1 + \frac{\varepsilon}{\rho(M)}$, so it takes $\Phi_\eta$ to a sweepout $\tilde{\Phi}_\eta$ of $N$ such that $L(\tilde{\Phi}_\eta) \leq \big( 1 + \frac{\varepsilon}{\rho(M)} \big) L(\Phi)$. It follows that $\beta(N) \leq \big(1 + \frac{\varepsilon}{\rho(M)} \big) \beta(M)$. Thus,
\[
\beta(N) - \beta(M) \leq \varepsilon \frac{\beta(M)}{\rho(M)}.
\]
In the same way, one may bound $\beta(M) - \beta(N)$ from above by $\varepsilon \frac{\beta(N)}{\rho(N)}$. \end{proof}

\begin{lemma}\label{birkhoff continuity}
The Birkhoff invariant is continuous on the set of closed, convex hypersurfaces in $\real^{n+1}$ with respect to the Hausdorff metric.
\end{lemma}

\begin{proof} Let $M_k^n$ be a sequence of closed, convex hypersurfaces in $\real^{n+1}$ such that $d_H(M_k,M) \to 0$, where $M$ is a closed, convex hypersurface in $\real^{n+1}$. For $0 < \varepsilon < \frac{1}{2}$, let $K$ be large enough that $d_H(M_k,M) \leq \varepsilon \rho(M)$ for all $k \geq K$. For all such $k$, $|\rho(M_k) - \rho(M)| \leq d_H(M_k,M) \leq \varepsilon \rho(M)$, where the first inequality follows from Lemma \ref{radius_continuity}. So, $\rho(M_k) \geq (1 - \varepsilon) \rho(M)$. Applying Lemma \ref{birkhoff inradius bound} twice shows that
\[
|\beta(M_k) - \beta(M)| \leq \varepsilon \cdot \max \{ \beta(M), 2 \beta(M_k) \} \leq 2\varepsilon(1+\varepsilon) \beta(M).
\]
It follows that $\beta(M_k) \to \beta(M)$. \end{proof}


\subsection{The Spherical Radon Transform}

\noindent Theorem \ref{Crofton's formula} implies that, whenever a convex hypersurface $M$ in $\real^{n+1}$ has constant width $w$, the boundary $\partial \phi_{\Pi}(M)$ of the projection of $M$ onto any $2$-plane $\Pi$ has length $\pi w$. A classical theorem of Minkowski \cite{Minkowski1904} gives a converse to this statement for convex surfaces in $\real^{3}$: if a convex surface $M$ has the property that, for any $2$-plane $\Pi$ in $\real^{3}$, the length of  $\partial \phi_{\Pi}(M)$ is equal to a constant $\xi$, then $M$ has constant width $\frac{\xi}{\pi}$. The aim of this subsection is to prove the following stable version of this result for convex hypersurfaces in all dimensions.

\begin{lemma}\label{nearly_constant_width_lemma}
Let $R$ and $\varepsilon$ be positive numbers. Then, there exists $\delta > 0$ such that the following holds: if $M^{n}$ is a closed, convex hypersurface in $\real^{n+1}$, $n \geq 2$, such that $R_c(M) \leq R$ and, for every $\Pi \in \mathrm{Gr}(2,n+1)$, $|L\left( \partial \phi_{\Pi}(M) \right) - \xi| < \delta$, then $w^{+}(M) - \varepsilon < \frac{\xi}{\pi} < w^{-}(M) + \varepsilon$. In particular, if a convex hypersurface $M$ in $\real^{n+1}$ has $L\left( \partial \phi_{\Pi}(M) \right) = \xi$ for all $\Pi \in \mathrm{Gr}(2,n+1)$, then $M$ has constant width $\frac{\xi}{\pi}$. 
\end{lemma}

\noindent The proof of Lemma \ref{nearly_constant_width_lemma} uses the $1$-dimensional spherical Radon transform. This is the linear map $f \mapsto \hat{f}$ taking continuous $f : S^n \to \real$ to $\hat{f} : \mathrm{Gr}(2,n+1) \to \real$, where $\hat{f} : \mathrm{Gr}(2,n+1) \to \real$ is defined via the identification of $\Pi \in \mathrm{Gr}(2,n+1)$ with the great circle $\zeta = \Pi \cap S^n$ by
\[
\hat{f}(\zeta) = \int_\zeta f(\theta) \,d\theta. 
\]
The kernel of this transformation is the set of odd maps \cite[Theorem 3.1.7]{Helgason1999}, so it is injective on even maps.

\begin{lemma}\label{radon transform}
For any positive numbers $R$ and $C$, let
\[
\mathscr{L}^n_{R,C} = \{ f : S^n \to \real \,|\, f \textrm{ is even and } R \textrm{-Lipschitz and } |f(x)| \leq C \textrm{ for all } x \in S^n \}.
\]
Then, $\mathscr{L}^n_{R,C}$ is compact in the uniform norm $\| \cdot \|_{\infty}$, and the restriction of the Radon transform to $\mathscr{L}^n_{R,C}$ is a homeomorphism onto its image (with respect to the topology given by the uniform norm in the domain and codomain).
\end{lemma}

\begin{proof} In the uniform norm, being even, $R$-Lipschitz, and bounded in norm by $C$ are preserved under limits. The compactness of $\mathscr{L}^n_{R,C}$ follows from a standard Arzel\`{a}--Ascoli argument. Since
\[
\| \hat{f} - \hat{g} \|_\infty = \max_{\zeta} \int_\zeta |f(\theta) - g(\theta)| \, d\theta \leq 2\pi \max_{x \in S^n} |f(x) - g(x)| = 2\pi \| f - g \|_\infty,
\]
continuity of the transform is clear. Since the transform is injective on the set of even functions, it is a bijection between $\mathscr{L}^n_{R,C}$ and its image. As a continuous bijection between a compact space and a Hausdorff space, it is a homeomorphism. \end{proof}

\noindent The following is an immediate corollary of Lemma \ref{radon transform}.

\begin{lemma}\label{inverse radon transform}
Let $R$, $C$, and $\varepsilon$ be positive numbers. Then, there exists $\delta > 0$ such that the following holds: if $f \in \mathscr{L}^n_{R,C}$ satisfies $|\int_\zeta f(\theta) d\theta| < \delta$ for every great circle $\zeta$, then $|f(x)| < \varepsilon$ for all $x \in S^n$.
\end{lemma}

\begin{proof}[Proof of Lemma \ref{nearly_constant_width_lemma}] Let $R$ and $\varepsilon$ be positive. According to Lemma \ref{inverse radon transform}, there exists $0 < \delta \leq \pi$ such that, for all $f \in \mathscr{L}^n_{2R,4R+1}$, $\| f \|_\infty < \varepsilon$ whenever $|\int_\zeta f(\theta)\,d\theta| < 2\delta$ for every great circle $\zeta$ in $S^n$. Let $M^n$ be any closed, convex hypersurface in $\real^{n+1}$ such that $R_c(M) \leq R$ and, for every $\Pi \in \mathrm{Gr}(2,n+1)$, $|L(\partial \phi_\Pi(M)) - \xi| < \delta$. For any $\Pi$, $\phi_\Pi(M)$ is contained in a circle of radius $R_c(M)$, so $L(\partial \phi_\Pi(M)) \leq 2\pi R_c(M)$ and, consequently, $\frac{\xi}{\pi} < 2R_c(M) + 1$. Thus, the function $w(M,\cdot) - \frac{\xi}{\pi}$ is even, $2R_c(M)$-Lipschitz continuous, and bounded above by $4R_c(M) + 1$. It follows that $\| w(M,\cdot) - \frac{\xi}{\pi} \|_\infty < \varepsilon$ whenever $|\int_\zeta [w(M,\theta) - \frac{\xi}{\pi}] \,d\theta| < 2\delta$ for every $\zeta = \Pi \cap S^n$. By Theorem \ref{Crofton's formula},
\[
\frac{1}{2} \Big| \int_\zeta \Big[ w(M,\theta) - \frac{\xi}{\pi} \Big] d\theta \Big| = |\pi \Xi(\partial \phi_\Pi(M)) - \xi| = |L(\partial \phi_\Pi(M)) - \xi|,
\]
proving the result. \end{proof}



\begin{proof}[{\em 4.3.} \textbf{\em Proof of Theorem \ref{mean_width_theorem}}] Let $M^{n}$ be a closed, convex hypersurface in $\real^{n+1}$, $n \geq 2$. Suppose first that $M$ is strictly convex and, therefore, intersects all affine $2$-planes in a nontrivial convex curve, a point, or the empty set. For each $\Pi \in \mathrm{Gr}(2,n+1)$, one may define a sweepout of $M$ in the following way: the image $\phi_{\Pi^\perp}(M)$ of the orthogonal projection of $M$ into $\Pi^\perp$ is topologically an $(n-1)$-ball; map each $x \in \phi_{\Pi^\perp}(M)$ to $\Pi_x \cap M$, where $\Pi_x$ is the $2$-plane through $x$ parallel to $\Pi$. Therefore, 
\begin{align}\label{mean_width_pf_eqn_1}
\displaystyle \beta(M) \leq \max\limits_{\Pi'} L\left( \Pi' \cap M \right) \leq L\left( \partial \phi_{\Pi}(M) \right),
\end{align}
where $\Pi'$ varies over all $2$-planes parallel to $\Pi$. Note that the second inequality in \eqref{mean_width_pf_eqn_1} follows from Corollary \ref{crofton_corollary}. Integrating \eqref{mean_width_pf_eqn_1} over $\mathrm{Gr}(2,n+1)$, drawing on Lemma \ref{integral_geometric_formula}, gives the inequality in \eqref{mean_width_thm_eqn}. Because mean width and the Birkhoff invariant are both continuous with respect to the Hausdorff metric by Lemmas \ref{width_continuity} and \ref{birkhoff continuity}, the inequality in \eqref{mean_width_thm_eqn} follows for all convex hypersurfaces from the strictly convex case and the fact that every closed, convex hypersurface can be approximated arbitrarily closely in the Hausdorff metric by strictly convex hypersurfaces, as shown, for example, in \cite[pp. 158--160]{Schneider1993}. 

To establish the equality condition, let $M^n$ be a closed, convex hypersurface in $\real^{n+1}$ for which $\beta(M) = \pi \Xi(M)$, and let $M_k^n$ be a sequence of strictly convex hypersurfaces converging to $M$ in the Hausdorff metric. Lemma \ref{width_continuity} implies that $w(M_k,\cdot)$ converges to $w(M,\cdot)$ in the uniform norm $\| \cdot \|_{\infty}$, Lemma \ref{radius_continuity} implies that $R_c(M_k) \to R_c(M)$, and Lemma \ref{birkhoff continuity} implies that $\beta(M_k) \to \beta(M)$. For all $\Pi \in \mathrm{Gr}(2,n+1)$, the fact that orthogonal projection does not increase distance implies, as in the proof of Lemma \ref{width_continuity}, that $d_H(\partial \phi_{\Pi}(M_k),\partial \phi_{\Pi}(M)) \leq d_H(M_k,M)$. So, by Lemma \ref{width_continuity}, the function $\Pi \mapsto L(\partial \phi_{\Pi}(M_k))$ converges uniformly to $\Pi \mapsto L(\partial \phi_{\Pi}(M))$ on $\mathrm{Gr}(2,n+1)$. Because $\Xi(M_{k}) \to \Xi(M) = \frac{\beta(M)}{\pi}$, and because $\Xi(M_{k})$ is given by the integral formula in Lemma \ref{integral_geometric_formula}, the uniform convergence of $L(\partial \phi_{\Pi}(M_k))$ to $L(\partial \phi_{\Pi}(M))$, together with \eqref{mean_width_pf_eqn_1} and the convergence of $\beta(M_k)$ to $\beta(M)$, implies that $L(\partial \phi_{\Pi}(M)) = \beta(M)$ for all $\Pi \in \mathrm{Gr}(2,n+1)$. Therefore, by Lemma \ref{nearly_constant_width_lemma}, $M$ has constant width $\frac{\beta(M)}{\pi}$. 

We next show that $M$ has long diametrical slices. This implies by Theorem \ref{slicing property implies round} that $M$ is a round sphere and completes the proof. Let $x$ and $y$ be any pair of diametrically opposed points of $M$ and $\Pi$ any $2$-plane containing $x$ and $y$, let $\varepsilon > 0$, and choose $\delta > 0$ so that the conclusion of Lemma \ref{width stability} holds for $R = 2R_c(M)$ and $\varepsilon$. For all $k$, we have $\beta(M_k) \leq \max_{\Pi'} L(\Pi' \cap M_k)$, where the maximum is taken over all $2$-planes $\Pi'$ parallel to $\Pi$. Because $\beta(M_k) \to \beta(M) = \pi w^{+}(M) = \lim\limits_{k \to \infty} \pi w^{+}(M_{k})$, letting $\Pi_k'$ be a sequence of $2$-planes realizing this maximum, we have that $L(\Pi_k' \cap M_k) > \pi w^{+}(M_k) - \delta$ for all sufficiently large $k$. Therefore, by Lemma \ref{width stability}, $\| w(\Pi'_k \cap M_k, \cdot) - w^{+}(M_k) \|_\infty < \varepsilon$. Because $w^{+}(M_k)$ converges to the constant width $\frac{\beta(M)}{\pi}$ of $M$, and because $\varepsilon$ is arbitrary, this implies the width functions $w(\Pi'_k \cap M_k, \cdot)$ converge uniformly to $\frac{\beta(M)}{\pi}$ and, by Theorem \ref{Crofton's formula}, that $L(\Pi'_k \cap M_k) \to \beta(M)$.

We claim that $d_H(\Pi_k', \Pi) \to 0$ as $k \to \infty$. To show this, let $x_k,y_k \in \Pi_k' \cap M_k$ be points such that $\langle y_k - x_k, \frac{y - x}{|y - x|} \rangle = w \big( \Pi_k' \cap M_k, \frac{y - x}{|y - x|} \big)$. We then have $w \big( \Pi_k' \cap M_k, \frac{y - x}{|y - x|} \big) \leq |y_k - x_k| \leq w^{+}(\Pi_k' \cap M_k)$. Because $w(\Pi'_k \cap M_k, \cdot)$ converges uniformly to $\frac{\beta(M)}{\pi}$, this implies that $|y_k - x_k| \to \frac{\beta(M)}{\pi}$, that $\langle y_k - x_k, \frac{y - x}{|y - x|} \rangle \to \frac{\beta(M)}{\pi}$, and, therefore, that $\frac{y_k - x_k}{|y_k - x_k|} \to \frac{y - x}{|y - x|}$. Suppose $x^*,y^* \in M$ are the limits of any convergent subsequences of $x_k$ and $y_k$, respectively. Since $|y^* - x^*| = \langle y^* - x^*, \frac{y - x}{|y - x|} \rangle = \frac{\beta(M)}{\pi} = w^{+}(M)$ and the diametrical segment in each direction in a hypersurface of constant width is unique, $y^* = y$ and $x^* = x$. Since each $\Pi_k'$ is parallel to $\Pi$, the claim follows. 

Let $M_{\varepsilon}$ be the outer parallel hypersurface to $M$ at distance $\varepsilon > 0$. After identifying $\Pi_k'$ and $\Pi$ via parallel translation, the sequence $\Pi_k' \cap M_k$ of strictly convex plane curves is contained within the closed, convex curve $\partial \phi_{\Pi}(M_{\varepsilon})$ for all sufficiently large $k$. The Arzel\`{a}--Ascoli theorem implies that any subsequence of $\Pi_k' \cap M_k$ contains a convergent subsequence. The limit of any convergent subsequence of $\Pi_k' \cap M_k$ must be a planar convex curve, must be contained within $\Pi$, because $\Pi_k' \to \Pi$, and must be contained within $M$, because $M_k \to M$; it follows that $d_H(\Pi_k' \cap M_k,\Pi \cap M) \to 0$ and, therefore, that $L(\Pi \cap M) = \lim\limits_{k \to \infty} L(\Pi_k' \cap M_k) = \beta(M)$. Because $\Pi \cap M$ was an arbitrary diametrical slice of $M$, this implies that all diametrical slices of $M$ have length $\beta(M) = \pi w^{+}(M)$. \end{proof}


\section{Mean curvature and closed geodesics}
\label{curvature_section}


In this section, we prove Theorem \ref{mean_curvature_theorem} as a special case of a lower bound for a broader class of curvature functionals in Proposition \ref{main_curvature_prop} and Remark \ref{main_thm_remark}. Given an $n$-tuple of real numbers $a_1,a_2,\dots,a_n$, denote by $S_{k}(a)$ the $k$-th elementary symmetric function of $a_1,a_2,\dots,a_n$. That is,
\begin{align}\label{sym_fcn}
\displaystyle S_{k}(a) = \frac{\sum\limits_{1 \leq i_{1} < i_{2} < \cdots < i_{k} \leq n}a_{i_{1}}a_{i_{2}}\cdots a_{i_{k}}}{\binom{n}{k}}. 
\end{align}
\textbf{Maclaurin's inequalities} \cite{HardyLittlewoodPolya1952} state that, for $n$-tuples of nonnegative real numbers, 
\begin{align}
\displaystyle S_{1}(a) \geq \sqrt{S_{2}(a)} \geq \sqrt[3]{S_{3}(a)} \geq \cdots \geq \sqrt[n]{S_{n}(a)},  
\end{align}
with equality in any of these inequalities if and only if $a_{1} = a_{2} = \dots = a_{n}$. A basic integral-geometric identity \cite{Schneider1993} gives the following expression for the mean width of a smooth, closed, convex hypersurface $M^{n}$ in $\real^{n+1}$ with principal curvatures $\lambda_{1}, \lambda_{2}, \dots, \lambda_{n}$: 
\begin{align}\label{width_curvature_identity}
\displaystyle \Xi(M) = \frac{2}{\sigma_{n}} \int_{M} S_{n-1}(\lambda) \, dA. 
\end{align}

\begin{proposition}\label{main_curvature_prop}
Let $M^{n}$ be a smooth, closed, convex hypersurface in $\real^{n+1}$, $n \geq 2$, with principal curvatures $\lambda_{1}, \lambda_{2}, \dots, \lambda_{n}$, and let $\Lambda(M)$ be the length of a shortest nontrivial closed geodesic in $M$. Then, for $1 \leq m \leq n-1$, 
\begin{align}\label{mcp_eqn}
\displaystyle \int_{M} S_{m}(\lambda)^{\frac{n-1}{m}} dA \geq \frac{\sigma_n}{2\pi} \Lambda(M).
\end{align}
Equality holds if and only if $M$ is a round sphere. 
\end{proposition}

\begin{proof} Theorem \ref{mean_width_theorem} and the fact that $\Lambda(M) \leq \beta(M)$ imply that
\begin{align}\label{mcp_pf_eqn_1}
\displaystyle \Xi(M) \geq \frac{1}{\pi} \Lambda(M). 
\end{align}
For $m = n-1$, the inequality in \eqref{mcp_eqn} follows immediately from \eqref{mcp_pf_eqn_1} and \eqref{width_curvature_identity}. For $n \geq 3$ and $n-1 > m \geq 1$, the inequality follows from the result for $n-1$ and Maclaurin's inequalities. Equality requires equality in Theorem \ref{mean_width_theorem}, which implies that $M$ is round. \end{proof}

\begin{remark}\label{main_thm_remark}
Theorem \ref{mean_curvature_theorem} is the case $m = 1$ of Proposition \ref{main_curvature_prop}. The scalar curvature of $M$ is given by $S_{2}(\lambda)$, so, for $n \geq 3$, the case $m = 2$ of Proposition \ref{main_curvature_prop} gives a sharp lower bound for the integral of the scalar curvature of a convex hypersurface. For $n \geq 3$ and $n-1 > m \geq 1$, equality in Proposition \ref{main_curvature_prop} implies, by the equality conditions in Maclaurin's inequalities \eqref{sym_fcn}, that $M$ is totally umbilic, i.e., that the principal curvatures of $M$ at each point coincide. This implies that $M$ is round, as discussed in Remark \ref{blaschke_remark}. In this sense, the characterization of equality in Proposition \ref{main_curvature_prop} in these cases does not depend on the characterization of equality in Theorem \ref{mean_width_theorem}. 

Let $\mathrm{Gr}(n,n+1)$ denote the Grassmannian of hyperplanes $\Pi^n$ in $\real^{n+1}$. The mean curvature of a smooth convex hypersurface $M^n$ in $\real^{n+1}$ can be calculated by integrating the $(n-1)$-dimensional measure of $\partial \phi_{\Pi^n}(M)$ over $\mathrm{Gr}(n,n+1)$, generalizing the formula for the mean curvature of a convex surface given by combining Lemma \ref{integral_geometric_formula} with \eqref{width_curvature_identity}, cf. \cite{Schneider1993}. The proof of the inequality in Theorem \ref{mean_width_theorem} can therefore be adapted to give a lower bound for $\int_M \mathcal{H} \ dA$ in terms of the $(n-1)$-dimensional min-max area of a sweepout of $M$. This min-max invariant is not known to coincide with the area of a minimal hypersurface in $M$ for $n\geq 4$, but Smith \cite{Smith1983} showed that every Riemannian metric on the $3$-sphere has an embedded minimal $2$-sphere given by this min-max construction. Smith's result, together with the argument above, therefore gives a sharp lower bound for the total mean curvature of a smooth, convex $3$-sphere $M$ in $\real^{4}$ in terms of the least area of a minimal $2$-sphere in $M$. 

The Gauss-Kronecker curvature $S_{n}(\lambda)$ is the determinant of the differential of the Gauss map. The integral of $S_{n}(\lambda)$ is therefore equal to $\sigma_{n}$ and does not admit a lower bound in terms of $\Lambda(M)$ or any other geometric invariant. However, applying Maclaurin's inequalities to $\int_{M} S_{n}(\lambda) \ dA$, as in the proof of Proposition \ref{main_curvature_prop}, gives a family of sharp lower bounds for the integrals of scale-invariant powers of $S_{1}(\lambda),S_{2}(\lambda),\dots,S_{n-1}(\lambda)$, including the following special case of the results of Chen and Willmore cited above \cite{Willmore1965,Chen1971b}.
\end{remark}

\begin{theorem}[{Chen--Willmore, see also \cite[Corollary 1.2]{CederbaumMiehe2025}}]\label{chen_willmore}
If $M^n$ is a smooth, closed, convex hypersurface in $\real^{n+1}$, $n \geq 2$, then $\int_M \mathcal{H}^n dA \geq \sigma_n$, with equality if and only if $M$ is a round sphere.
\end{theorem}

\noindent Theorem \ref{mean_curvature_theorem} also gives the following relationships between area, Willmore energy, and the length of the shortest closed geodesic in convex surfaces.

\begin{proposition}\label{willmore_prop}
Let $M^2$ be a smooth, closed, convex surface of area $|M|$ in $\real^{3}$ and $\mathcal{W}(M) = \int_{M} (\lambda_{1}^{2} + \lambda_{2}^{2}) \ dA$ its {\bf Willmore energy}. Then,
\begin{flalign*}
    & & 8 \Lambda(M)^{2} &\leq |M| \mathcal{W}(M) & \\[ -0.5ex ]
    \text{and} & & & & \\[-0.5ex]
    & & 16 \Lambda(M)^{2} &\leq |M| \left[ \mathcal{W}(M) + 8\pi \right]. &
\end{flalign*}
Equality in either inequality holds if and only if $M$ is a round sphere. 
\end{proposition}

\begin{proof}
The case $n=2$ of Theorem \ref{mean_curvature_theorem} implies that
\begin{equation}\label{willmore_prop_pf_eqn_1}
\displaystyle 4\Lambda(M) \leq \int_{M} (\lambda_{1} + \lambda_{2}) \ dA.
\end{equation}	
Applying the Cauchy--Schwarz inequality to \eqref{willmore_prop_pf_eqn_1} pointwise implies that $2\sqrt{2}\Lambda(M) \leq \int_{M} \sqrt{\lambda_{1}^{2} + \lambda_{2}^{2}} \ dA$, and another application of Cauchy--Schwarz gives the first inequality. Applying Cauchy--Schwarz directly to the integral in \eqref{willmore_prop_pf_eqn_1} gives $16\Lambda(M)^{2} \leq |M| \int_{M} (\lambda_{1} + \lambda_{2})^{2} dA $. A standard Gauss--Bonnet calculation proves the second inequality. Equality in either requires equality in Theorem \ref{mean_curvature_theorem}.
\end{proof}

\noindent The first inequality in Proposition \ref{willmore_prop} would imply the second if the area of the round sphere were minimal in terms of the length of the shortest closed geodesic, but this is known not to be the case. Conjecturally, the optimal metric is singular, given by doubling a flat equilateral triangle, cf. \cite{Croke1988}. Without assuming convexity, M\"{u}ller, Rupp, and Scharrer \cite[Theorem 1.1]{MullerRuppScharrer2026} recently proved a universal inequality relating $|M|$, $\mathcal{W}(M)$, and $\Lambda(M)$ for immersed $2$-spheres in $\real^n$ satisfying $\mathcal{W}(M) < 6\pi$.

The following inequality of Croke can also be derived from Proposition \ref{main_curvature_prop}.

\begin{theorem}[{\cite[Theorem 1.7]{Croke1982}}]\label{croke_theorem}
Let $M$ be a smooth, closed, convex hypersurface in $\real^{n+1}$, $n \geq 2$, and $S_{n-1}(\lambda^{2})$ the $(n-1)$-st symmetric function of the squares of its principal curvatures. Then,
\begin{align}\label{croke_thm_eqn}
\displaystyle \int_{M} \sqrt{S_{n-1}(\lambda^{2})} \, dA \geq \frac{\sigma_{n}}{2\pi} \Lambda(M).
\end{align}
Equality holds if and only if $M$ is a round sphere.
\end{theorem}

\begin{proof} The result follows from the case $m = n-1$ of Proposition \ref{main_curvature_prop} and the pointwise inequality $\sqrt{S_{n-1}(\lambda^{2})} \geq S_{n-1}(\lambda)$, which follows from the Cauchy--Schwarz inequality. \end{proof}

\noindent Conversely, Croke notes in \cite[Corollary 1.8(b)]{Croke1982} that Theorem \ref{croke_theorem} implies a non-sharp version of Proposition \ref{main_curvature_prop}, with a constant off by a factor of $\sqrt{n}$ from the optimal constant. The characterization of equality in the proof of Theorem \ref{croke_theorem} in \cite{Croke1982} is based on the proof of the Blaschke conjecture, by an argument similar to the one described in Remark \ref{blaschke_remark}. 


\bibliographystyle{alpha}
\bibliography{bibliography}


\end{document}